\documentclass[a4paper]{article}       

\usepackage{etex}

\usepackage[margin=1in]{geometry}

\usepackage{amsmath}
\usepackage{amssymb}
\usepackage{amsthm}
\usepackage{gensymb}
\usepackage{graphicx}
\usepackage[hypcap=true]{caption}
\usepackage[hypcap=true]{subcaption}
\usepackage{placeins}	

\usepackage{mathtools}
\usepackage{multirow}
\usepackage{enumitem}
\usepackage{mathrsfs}

\usepackage{algorithm}
\usepackage[noend]{algpseudocode}
\algrenewcommand{\algorithmiccomment}[1]{\hskip\algorithmicindent \# #1}
\usepackage{setspace}

\usepackage{booktabs}
\usepackage{makecell}

\theoremstyle{plain}
\newtheorem{theorem}{Theorem}
\newtheorem{definition}{Definition}

\usepackage{hyperref}	
\hypersetup{colorlinks,breaklinks,
            urlcolor=black,
						citecolor=black,
            linkcolor=black}
\usepackage[all]{hypcap}
\usepackage[nameinlink,noabbrev]{cleveref}

\newcommand{\R}{\mathbb R}
\newcommand{\D}{\mathscr{D}}
\newcommand{\A}{\mathscr{A}}
\newcommand{\dd}{\mathrm{d}}
\newcommand{\T}{\mathcal{T}}
\newcommand{\SSS}{\mathscr{S}}
\newcommand{\pfreeB}{\langle \Gamma, \mathscr{S}\rangle}

\begin{document}

\title{Finite element approximations of minimal surfaces:
 algorithms and mesh refinement}

\author{
  Aymeric Grodet\thanks{Corresponding author} \thanks{E-mail addresses: aymeric.grodet@gmail.com,
    tsuchiya@math.sci.ehime-u.ac.jp.}
  \and
  Takuya Tsuchiya\footnotemark[2]
}
\date{Graduate School of Science and Engineering, \\Ehime University, 2-5, Bunkyo-cho, Matsuyama, Japan}

\maketitle

\begin{abstract}
Finite element approximations of minimal surface are not always precise. 
They can even sometimes completely collapse. 
In this paper, we provide a simple and inexpensive method, in terms of computational cost, 
to improve finite element approximations of minimal surfaces 
by local boundary mesh refinements. 
By highlighting the fact that a collapse is simply 
the limit case of a locally bad approximation, 
we show that our method can also be used to avoid the collapse of finite element approximations. 
We also extend the study of such approximations 
to partially free boundary problems and give a theorem for their convergence. 
Numerical examples showing improvements induced  by the method are given throughout the paper.

\end{abstract}

\section{Introduction}
\label{sec:intro}
Let $D = \left\{(u,v)\in\mathbb{R}^2 \,\middle|\, u^2+v^2<1\right\}$
 be the unit disk and $\partial D = S^1$ be its boundary.
Let $\varphi:\overline{D} \to \R^d$ $(d \ge 2)$ be a map
that is sufficiently smooth and $\operatorname{rank} D\varphi = 2$ 
almost everywhere in $D$, where $D\varphi$ is the Jacobi matrix of $\varphi$.
By this assumption, its image $\varphi(D) \subset \R^d$ is a
two-dimensional surface, possibly with self-intersections. 
In this paper, we refer to the map as $\varphi$ and to the surface as 
$\varphi(D)$.  If the mean curvature of $\varphi$ vanishes at each point
on $\varphi(D)$, then $\varphi(D)$ is called a \emph{minimal surface}.

Let $\Gamma \subset \R^d$ $(d \ge 2)$  be an arbitrary Jordan curve.
That is, $\Gamma$ is the image of a continuous embedding of $\partial D$
into $\R^d$. We would like to find minimal surfaces $\varphi$ spanned in
 $\Gamma$, that is, 
\begin{equation*}
  \varphi : \overline{D} \to \mathbb{R}^d
   \text{ with } \varphi(\partial D) = \Gamma.
\end{equation*}
For a given Jordan curve $\Gamma$, the problem of finding
minimal surfaces spanned in $\Gamma$ is called
the (\emph{classical}) \emph{Plateau problem}
\cite{Courant},  \cite{DHS}, \cite{Osserman}.
For the Plateau problem, the following variational principle
has been known \cite{Courant}, \cite{DHS}:\\
Define the subset
$X_\Gamma$ of $C(\overline{D};\R^d)\cap H^1(D;\R^d)$ by
\begin{equation}
 X_\Gamma:=\Bigl\{\psi\in C(\overline{D};\R^d)\cap
 H^1(D;\R^d)\Bigm| \psi(\partial D)= \Gamma\text{ and }
 \psi|_{\partial D} \text{ is monotone} \Bigr\},
 \label{eq1.1}
\end{equation}
where $\psi|_{\partial D}$ being monotone means that
$(\psi|_{\partial D})^{-1}(p)$ is connected for any $p\in\Gamma$. 
Although $\psi|_{\partial D}$ has to be onto, it does not need to be one-to-one.
We denote the Dirichlet integral (or the energy functional)
on $D$ for $\varphi=(\varphi_1, \cdots, 
\varphi_d)^\top\in H^1(D;\R^d)$ by
\begin{equation}
 \D(\varphi):= \frac12 \int_D|\nabla \varphi|^2 \dd x
 = \frac12 \sum_{i=1}^d\int_D |\nabla\varphi_i|^2 \dd x.
  \label{dirichlet_int}
\end{equation}
Then, \emph{$\varphi\in X_\Gamma$ is a minimal surface 
if and only if $\varphi\in X_\Gamma$ is a stationary point 
of the functional $\D$ in $X_\Gamma$.}
Moreover, we have
\begin{equation*}
 \operatorname{Area}(\varphi(D)) = \D(\varphi) = \inf_{\psi\in X_\Gamma}\D(\psi). 
\end{equation*}

To obtain numerical approximations for the Plateau problem, the
piecewise linear finite element method has been applied 
\cite{Tsuchiya1,Tsuchiya2,Tsuchiya3,Tsuchiya4,Tsuchiya5}.  Firstly, functions of $X_\Gamma$
are approximated by piecewise linear functions on a triangulation of $D$. 
Then, starting from a suitable
initial surface, stationary surfaces of the Dirichlet integral are
computed by a relaxation procedure. This method has the advantage to be
quite simple and straightforward to put in use.
In the relaxation procedure, the images of inner nodes are moved 
$d$-dimensionally by, for example, Gauss-Seidel method, and the images
of nodes on $\partial D$ are moved through $\Gamma$ by, for example,
Newton method. See \cite[Figure~1]{Tsuchiya1}. 
Because the images of nodes on $\partial D$ will move rather freely on
$\Gamma$, finite element approximations can be very poor, sometimes even 
collapse, if the configuration of $\Gamma$ is not simple enough.

The first aim of this paper is to provide a technique to overcome this
difficulty.  We will present a straightforward adaptive mesh 
refinement algorithm that is almost inexpensive in terms of computational cost.
The details of our refinement technique will be explained in Section~\ref{sec:fea}.

The second aim of this paper is to extend the results of finite element (FE) 
minimal surfaces to the case of minimal surfaces with \emph{partially
free boundary}.  Let a smooth surface $\SSS \subset \R^d$ be given.
Suppose that $\Gamma \subset \R^d$ is now a piecewise smooth curve with
end-points on $\SSS$.  That is, $\Gamma$ is the image of a smooth map
$\gamma:[-1,1] \to \R^d$ with $\gamma(-1), \gamma(1) \in \SSS$.
We would like to find minimal surfaces $\varphi$ such that
$\varphi(\partial D) \subset \Gamma \cup \SSS$.  Note that the image
of $\varphi$ on $\SSS$ is not known \emph{a priori}.  In Section~\ref{sec:pfb},
we will show that our methodology works well to obtain FE
minimal surfaces with partially free boundary.  We will also give a
theorem for convergence of FE minimal surfaces.

In Section~\ref{sec:data}, we will discuss some useful data structures and 
specify a general algorithm to compute FE approximations with the method 
that will have been discussed.

To highlight the effectiveness of the proposed method, several numerical examples will be given throughout this paper.

\section{Minimal surfaces}
\label{sec:minsurf}
\newcommand{\bfx}{\mathbf{x}}
As explained in the previous section, in this paper, \emph{surfaces} refer to mappings 
\[
   \varphi:\Omega\to\R^d, \qquad
    \varphi = (\varphi_1,\cdots,\varphi_d)^\top,
\]
from a bounded domain $\Omega\subset\R^2$ into $\R^d$
with $\operatorname{rank}D\varphi = 2$.
A point $p \in \Omega$  is always written as
$p=(u,v)^\top$. The \emph{area functional}  $\A(\varphi)$ of the surface
$\varphi$ is defined by
\begin{gather*}
    \A(\varphi)=\int_{\Omega}{}{\left|\varphi_u \wedge
\varphi_v\right|} \dd \bfx, \qquad
   \dd \bfx = \dd{u}\dd{v}, \\
  \varphi_u := \left(\frac{\partial \varphi_1}{\partial u}, \cdots,
   \frac{\partial \varphi_d}{\partial u}\right)^\top, \quad
  \varphi_v := \left(\frac{\partial \varphi_1}{\partial v}, \cdots,
   \frac{\partial \varphi_d}{\partial v}\right)^\top,
\end{gather*}
and stationary points of $\A$ are called \emph{minimal surfaces}.

The stationary points of the area functional, and in particular, its
minimizers, are the surfaces of zero mean curvature. In other
words, a map is a minimal surface if and only if its mean curvature
vanishes at any point on the surface. 

We consider the \emph{Plateau problem}. 
This problem is to find minimal surfaces mapping a bounded, 1-connected domain to a surface contoured by $\Gamma$. 
A domain is 1-connected (or simply-connected) is it is path-connected and if every path between two points can be continuously transformed into any other path with the same endpoints (informally, there is no ``hole''). 
We may assume without loss of generality that $\Omega$ is the unit disk.
In the sequel of this paper, we only consider the case $\Omega = D$.

The Belgian physicist Joseph Plateau made several
experiments with soap films. In particular, he observed that when he
dipped a frame consisting of a single closed wire into soapy water, it
would always result in a soap surface spanned in the closed wire,
whatever may be the geometrical form of the frame. From a mathematical
point of view, the single closed wire is a Jordan curve, a curve
topologically equivalent to the unit circle, and the resulting soap film
is a surface in $\R^3$. From the theory of capillary surfaces, we know
that the surface energy is proportional to its area. From these
observations, one has good reasons to think that every (rectifiable)
Jordan curve bounds at least one minimal surface. The Plateau problem
is then first to show the existence of such surfaces. The problem has been
solved for general contours by Douglas~\cite{Douglas} and
Rad\'{o}~\cite{Rado}.

Note that, for any $\psi \in X_\Gamma$, we have
\[
   \A(\psi) = \A(\psi\circ\eta) \quad
 \text{for an arbitrary diffemorphism }\;
  \eta:\overline{D} \to \overline{D}.
\]
Hence, if we consider the Plateau problem with respect to 
the area functional $\A$, we would have to deal with the space of all
diffeomorphisms on $\overline{\Omega}$.  This was the main reason
why the problem was so difficult to solve.

Later, the existence proof was significantly simplified by Courant. 
Courant pointed out that a map $\varphi:D \to \R^d$ is a minimal surface
if and only if it is a stationary point of the Dirichlet integral
\eqref{dirichlet_int} in $X_\Gamma$.
Note that if a map $\varphi:\overline{D} \to \R^d$ is stationary in
$X_\Gamma$, then it satisfies the following equations:
\begin{align*}
  \Delta\varphi = \left(\Delta
 \varphi_1,\cdots,\Delta\varphi_d\right)^\top = 0 \quad \text{ and }
 \left|\varphi_u\right|^2 = \left|\varphi_v\right|^2, \;
 \varphi_u\cdot\varphi_v = 0 \quad \text{ in } D.
\end{align*}
The second and third equations mean that $\varphi$ is 
\emph{isothermal}, or \emph{conformal}.

Courant showed the following theorem
\cite[Chapter~3]{Courant}, \cite[Main~Theorem, p.270]{DHS}:
\begin{theorem}
There exists a map $\varphi \in X_\Gamma$ that attains the infimum
of the Dirichlet integral in $X_\Gamma$, that is,
\begin{equation}
     \D(\varphi) = \inf_{\psi \in X_\Gamma} \D(\psi).
   \label{Douglas-Rado}
\end{equation}
This $\varphi$ is a solution to the Plateau problem.
\end{theorem}
The map which satisfies \eqref{Douglas-Rado} is called
the \emph{Douglas-Rad\'o solution}.

Note that, for any $\psi \in X_\Gamma$, we have
\[
   \D(\psi) = \D(\psi\circ\eta) \quad
 \text{for an arbitrary conformal map }\;
  \eta:\overline{D} \to \overline{D}.
\]
Hence, if we consider the Plateau problem with respect to 
the Dirichlet integral $\D$, we only need to deal with the space of
all conformal maps on $\overline{D}$. A conformal map
$\eta:\overline{D} \to \overline{D}$ is determined uniquely by a 
\emph{normalization condition}, which can be one of the following conditions:
\begin{itemize}
 \item Assigning the image of three points $p_i  \in \partial D$,
  $i = 1,2, 3$.
 \item Assigning the image of one inner point 
   $p_0 \in D$ and one boundary point $p_1 \in \partial D$.
 \item Assigning the image of one inner point $p_0 \in D$ and the direction
       of the derivative at $p_0$.
\end{itemize}
Thus, if we use one of these normalization conditions,
a surface $\varphi \in X_\Gamma$ is (locally) \emph{``fixed''}. 
This was the reason why Courant could simplify the Plateau problem so much.

In this paper, we deal with the first of these conditions, which can be applied to any dimension higher than two. 
For the sake of precision and notation, let us write it in more details, this time applied to our problem: \\
Take different points $p_1, p_2, p_3 \in \partial D$
 and $q_1, q_2, q_3 \in \Gamma$.  Then, impose
  $\varphi(p_i) = q_i ,\, i=1,2,3$. 
Figure~\ref{fig:ncond} illustrates this normalization condition.

\begin{figure}
\centering
	\includegraphics[width=0.75\textwidth]{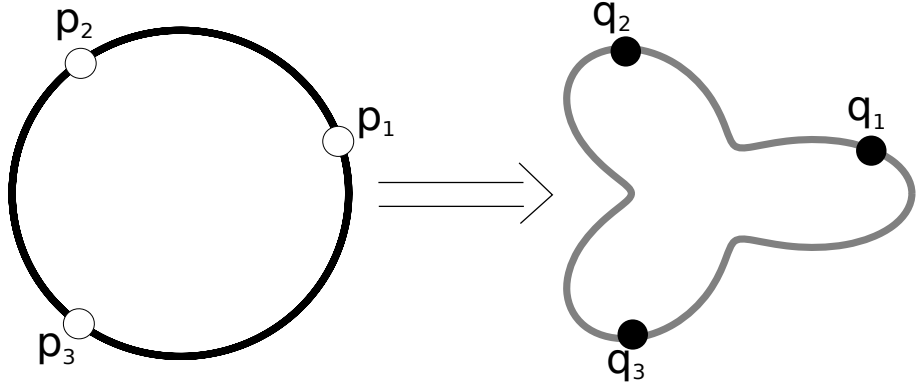}
	\caption{As a normalization condition, we fix the images of three points on the  boundary $(d \ge 2)$}
\label{fig:ncond}
\end{figure}

Let us summarize notation of function spaces we use in this paper.
The set of all continuous functions on $\overline{D}$ with the uniform
convergence metric is denoted by $C(\overline{D})$.
The spaces $L^2(D)$ and $H^1(D)$ are defined by 
\begin{align*}
 L^2(D) := \left\{f:D\to\R \,\middle|\,
 \int_{D}{}{|f|^2}{x}<\infty\right\}, \\
 H^1(D) := \left\{f\in L^2(D) \,\middle|\, f_u, f_v \in
 L^2(D)\right\}.
\end{align*}
Their norms are also defined as usual. 
The set of maps $\varphi:D \to \R^d$ whose components belong, for
example, to $L^2(D)$ is denoted by $L^2(D;\R^d)$.
The sets $C(\overline{D};\R^d)$ and $H^1(D;\R^d)$ are defined similarly.

\FloatBarrier

\section{Finite element approximation}
\label{sec:fea}
In this section, we detail how to compute a FE approximation
of the solution to the problem detailed in Section~\ref{sec:minsurf}. We
present a general method, highlight two problems that easily occur and
propose a solution.

\subsection{Definitions}
Let $\T$ be a \emph{face-to-face} triangulation of the unit disk $D$.
This means that $\T$ is a set of triangles (which are regarded as
closed sets) such that
\begin{itemize}
 \item $\displaystyle D_h := \bigcup_{K \in \T}K \subset \overline{D}$.
 \item If $K_1$, $K_2 \in \T$ with $K_1 \neq K_2$, then either
$K_1 \cap K_2 = \emptyset$ or $K_1 \cap K_2$ is a common vertex
or a common edge of $K_1$ and $K_2$.
\item $\partial D_h$ is a piecewise linear inscribed curve of 
$\partial D$.
\end{itemize}
For a triangulation $\T$, we define its fineness by
$|\T| := \max_{K \in \T} \operatorname{diam} K$. 

Let $\mathcal{P}^1$ be the set of all polynomials with two variables
of degree at most~$1$.  We introduce the set of piecewise linear
functions on $\T$ as
\begin{align*}
   S_h := \left\{v_h \in C(D_h) \bigm| v_h|_K \in \mathcal{P}^1,
  \forall K \in \T \right\}.
\end{align*}
Note that each $v_h \in S_h $ is defined only on
$D_h \subset \overline{D}$. 
We extend each function $v_h \in S_h$ to $D\backslash D_h$ by the way described in \cite{Tsuchiya1}. 
That yields the inequalities
\begin{equation}
  \|v_h\|_{H^1(D_h)} \le \|v_h\|_{H^1(D)} \le 
  (1 + Ch)\|v_h\|_{H^1(D_h)}.
\label{eq:ineq}
\end{equation}

Let $\mathcal{N}_{bdy}$ be the set of all nodes on $\partial D$.
We define the discretization of $X_\Gamma$ by 
\begin{align*}
 X_{\Gamma_h} :=\Bigl\{\psi_h \in (S_h)^d
  \Bigm| \psi_h(\mathcal{N}_{bdy}) \subset \Gamma, \;
   \psi_h|_{\partial D} \text{ is } d\text{-monotone} \Bigr\},
\end{align*}
where $\psi_h \in (S_h)^d$ being $d$-monotone means that
if $\partial D$ is traversed once in the positive direction, then $\Gamma$ is traversed once in a given direction, although we allow arcs of $\partial D$ to be mapped onto single points of $\Gamma$.

The triangulations we use for the different mappings of this paper are shown in Figure~\ref{fig:disk}. 
This figure also gives which points of the triangulation are chosen to play the role of 
the points used as a normalization condition. 
For simplicity, we shall refer to these points as \emph{fixed points}.
Triangulation~1 has 169 interior nodes and 48 boundary nodes. Triangulation~2 has 331 interior nodes 
and 66 boundary nodes. Except for one exception in Figure~\ref{fig:3d}, 
where we use $p_1$, $p'_2$, and $p'_3$, we always use as 
fixed points $p_1$, $p_2$, and $p_3$ such as indicated on the figures.
\begin{figure}
	\centering
	\subcaptionbox{Triangulation 1\label{fig:disk1}}
		{\includegraphics[width=0.4\textwidth]{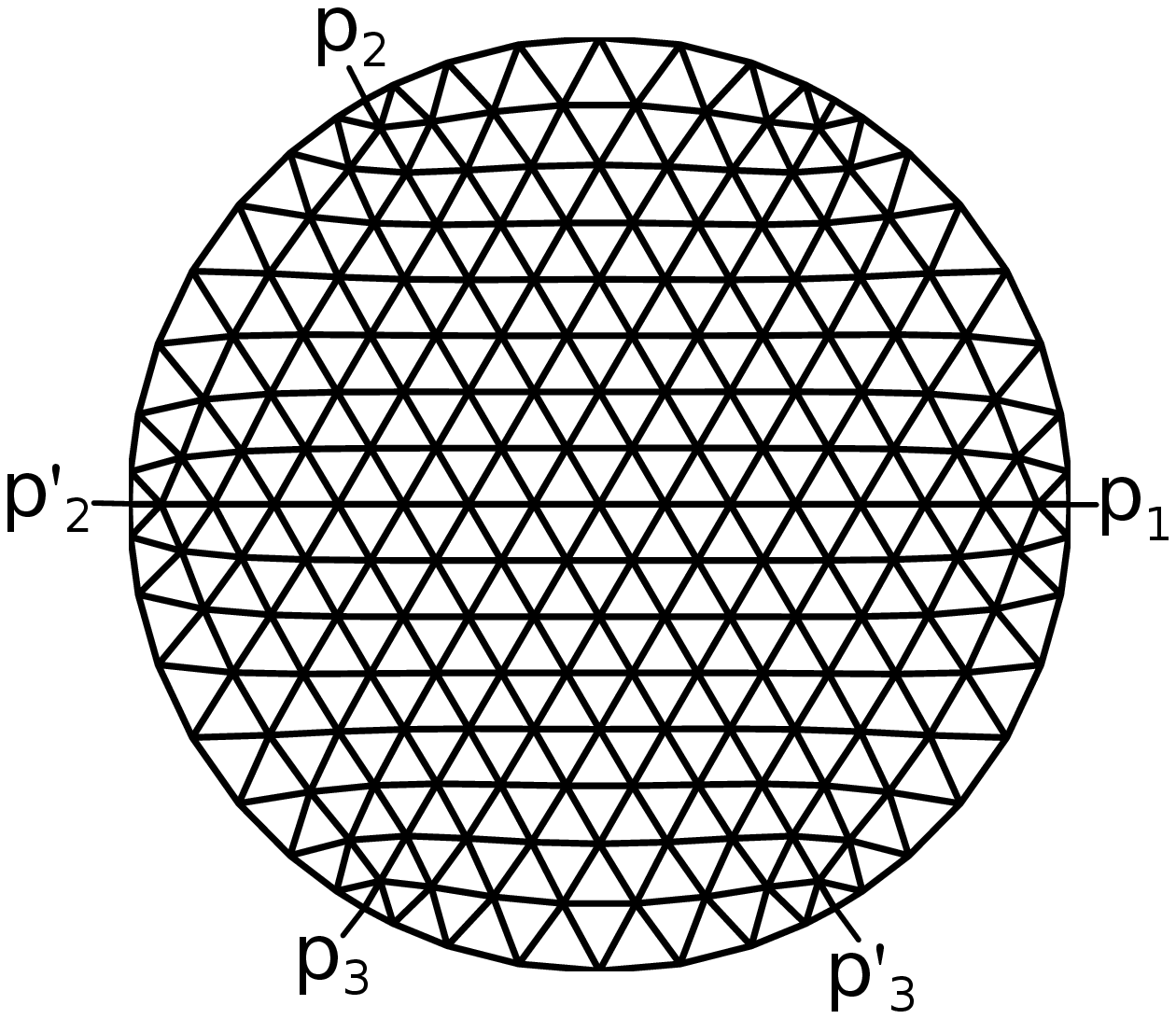}\quad}
	\subcaptionbox{Triangulation 2\label{fig:disk2}}
	{\quad\includegraphics[width=0.4\textwidth]{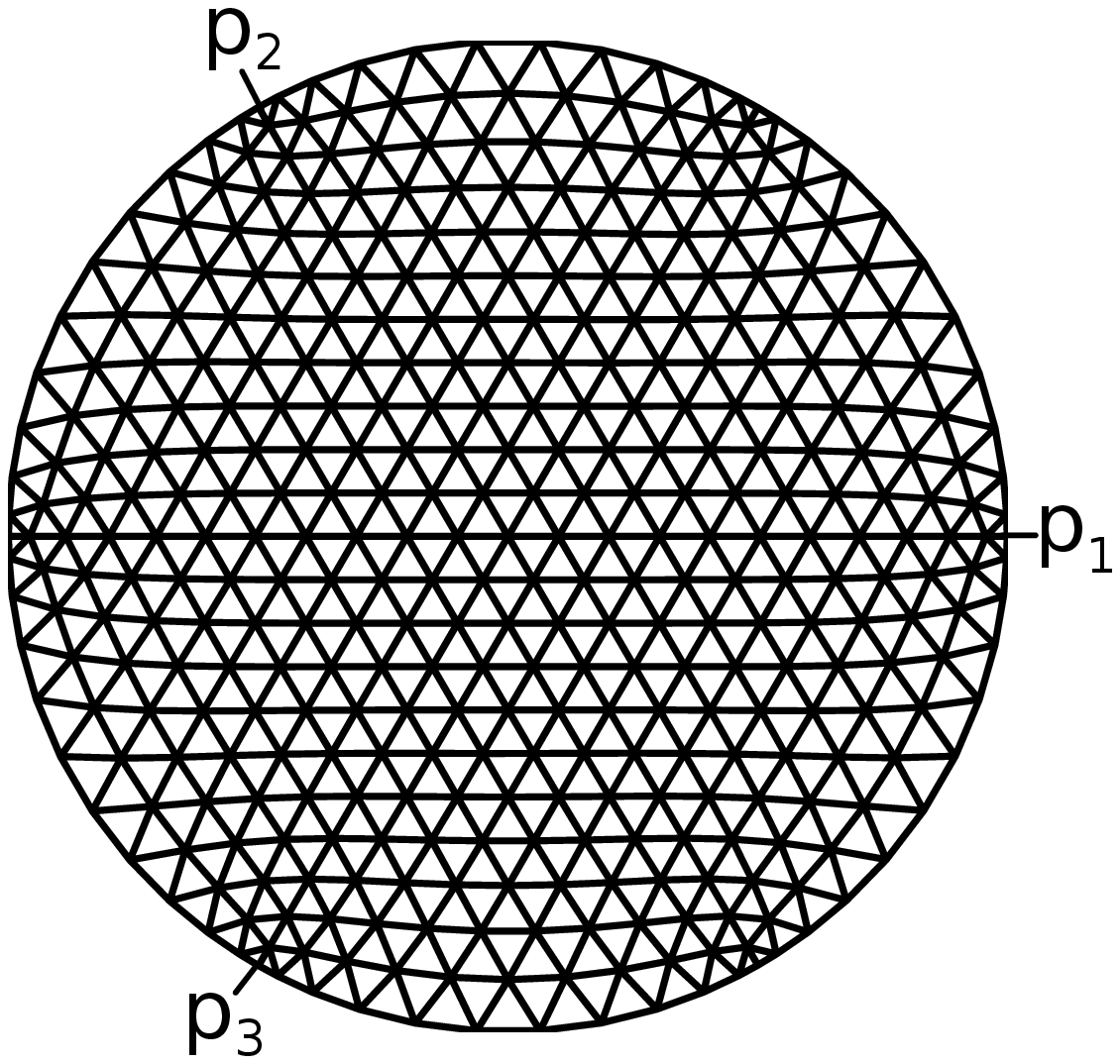}}
\caption{The two triangulations used in this paper}
\label{fig:disk}
\end{figure}

\begin{definition}
  A stationary point $\varphi_h  \in X_{\Gamma_h}$ of $\D$ is called
a \textbf{FE minimal surface} spanned in $\Gamma$.
In particular, a minimal surface $\varphi_h \in X_{\Gamma_h}$
such that
\begin{equation*}
    \D(\varphi_h) = \inf_{\psi_h \in X_{\Gamma_h}} \D(\psi_h)
\end{equation*}
is called \textbf{FE Douglas-Rad\'o solution}.
\end{definition}
From the definitions it is obvious that such solutions exist \cite[Section~5]{Tsuchiya2}.

To apply the three points condition, we take $p_i \in \partial D$ and
$q_i \in \Gamma$, $i=1,2,3$, and fix them.  In this paper, we always
assume that $p_i \in \mathcal{N}_{bdy}$.   We then define
$X_{\Gamma}^{tp}$ and $X_{\Gamma_h}^{tp}$ by
\begin{align*}
 X_{\Gamma}^{tp} & := \Bigl\{\psi\in X_{\Gamma}
  \Bigm| \psi(p_i) = q_i, \, i= 1,2,3\Bigr\}, \\
 X_{\Gamma_h}^{tp} & := \Bigl\{\psi\in X_{\Gamma_h}
  \Bigm| \psi(p_i) = q_i, \, i= 1,2,3\Bigr\}.
\end{align*}
The suffix ``$tp$'' stands for the \emph{three points condition}.

\subsection{Convergence}
For the convergence of FE minimal surfaces, the following
theorems have been known \cite{Tsuchiya2}, \cite{Tsuchiya3},
\cite{Tsuchiya4}.
Let $\{\T_h\}_{h>0}$ be a sequence of regular and quasi-uniform
triangulations \cite{Ciarlet} of $D$ such that $\lim_{h \to 0} |\T_h| = 0$, and
$p_i$, $i=1,2,3$ are nodal points for any $\T_h$.

\begin{theorem}\label{thm:conv}
Suppose that $\Gamma$ is rectifiable.
Let $\{\varphi_h\}_{h >0}$ be the sequence of FE Douglas-Rad\'o
solutions on $\T_h$.  Then, there exists a subsequence
$\{\varphi_{h_i}\} \subset X_{\Gamma_{h_i}}^{tp}$ which converges to one
of the Douglas-Rad\'o solution $\varphi \in X_{\Gamma}^{tp}$ spanned in
$\Gamma$ in the following sense$:$
\begin{equation}
    \lim_{h_i \to 0} \|\varphi_{h_i} - \varphi\|_{H^1(D;\R^d)}
     = 0,
   \label{convergence1}
\end{equation}
and if $\varphi \in W^{1,p}(D;\R^d)$, $p > 2$, then
\begin{equation}
    \lim_{h_i \to 0} \|\varphi_{h_i} - \varphi\|_{C(\overline{D};\R^d)}
     = 0.
   \label{convergence2}
\end{equation}
\end{theorem}

A minimal surface $\varphi \in X_\Gamma^{tp}$ is said to be
\emph{isolated and stable} if there exists a constant $\delta$
such that
\[
   0 < \|\varphi - \psi\|_{C(\overline{D};\R^d)} < \delta
  \quad \text{ implies } \quad
  \D(\varphi) < \D(\psi) \quad \text{ for }
   \psi \in X_\Gamma^{tp}.
\]

\begin{theorem}\label{thm:conv2}
Suppose that $\Gamma$ is rectifiable.
Let $\varphi \in X_\Gamma^{tp}$ be an isolated and stable minimal
surface spanned in $\Gamma$.  Then, there exists a sequence
$\{\varphi_h\}_{h > 0}$ of stable FE minimal surfaces
which converges to $\varphi$ in the sense of 
\eqref{convergence1} and \eqref{convergence2}. 
\end{theorem}

Remark~1: Note that in \cite{Tsuchiya2,Tsuchiya3,Tsuchiya4}, Theorems~2 and 3 were proved
under the assumption that the triangulations $\{\T_h\}$ are
regular, quasi-uniform, and \emph{non-negative type}.
As was pointed out in \cite{Tsuchiya4}, however, the weak discrete maximum
principle shown by Schatz \cite{Schatz} holds for discrete harmonic functions
(maps) on regular, quasi-uniform triangulations \cite[Lemma~2.1]{Tsuchiya4}.
Therefore, we here do not need to assume non-negative-typeness of
triangulations to show Theorems~2 and 3.

Dziuk and Hutchinson gave an error analysis of FE
minimal surfaces under certain regularity assumptions on $\Gamma$.
They claim that if $\varphi$ is a ``nondegenerate'' minimal surface
spanned in $\Gamma$, then there exist FE minimal surfaces
$\varphi_h \in X_\Gamma^{tp}$ such that
\[
   \|\varphi - \varphi_h\|_{H^1(D;\R^d)} \le C h,
\]
where $C$ is a constant independent of $h$.  See
\cite{DH1}, \cite{DH2} for the details.

\subsection{Boundary approximation of FE minimal surfaces}
%
Let $\T$ be a face-to-face triangulation of $D$ and $\{(x_i,y_i)\}_{i=1}^N$ be the set of nodes of $\T$, where $N$ is the number of nodes.
Let $\{\eta_i\} \subset S_h$ be the basis of $S_h$ defined by $\eta_i(x_i,y_i) = 1$ and $\eta_i(x_j,y_j) = 0$, for $i \neq j$.
Then, a piecewise linear surface $\psi_h$ is expressed by $\psi_h = \sum_{i=1}^N \eta_i(a_{i,1}, \cdots, a_{i,d})$,
where $(a_{i,1}, \cdots, a_{i,d}) \in \R^d$ is the image of the point $(x_i,y_i)$ by $\psi_h$.

Moreover, its Dirichlet integral $\D(\psi_h)$ is written as
\begin{align}
   \D(\psi_h) = \frac{1}{2} \int_D |\nabla \psi_h|^2 \dd x
  = \frac{1}{2}\left(a_1^\top \widetilde{A}a_1 + \cdots
   + a_d^\top \widetilde{A}a_d\right),
\end{align}
where
\begin{gather*}
   a_k = (a_{1,k}, \cdots, a_{N,k})^\top \in \R^N,\quad
   k = 1, \cdots, d, 
  \\
  \widetilde{A} = \left(\alpha_{ij}\right)_{i,j=1,\cdots,N}, \quad
  \alpha_{ij} := \int_D \nabla \eta_i\cdot \nabla \eta_j \dd x.
\end{gather*}
We would like to find a stationary point $\varphi_h$ of $\D$ in $X_{\Gamma_h}^{tp}$. 
Recall that, if $\psi_h \in X_{\Gamma_h}^{tp}$, then  $\psi_h(\mathcal{N}_{bdy}) \subset \Gamma$.
This means that, if $(x_i,y_i) \in \mathcal{N}_{bdy}$, $\psi_h(x_i,y_i)$ should be on $\Gamma$. 
Suppose that $\Gamma$ is parametrized by a parameter $t \in [0, 2\pi]$ as $\Gamma(t) = (\gamma_1(t), \cdots, \gamma_d(t))$. 
Then, if $\psi_h(x_i, y_i)\in \Gamma$, $\psi_h(x_i, y_i)$ is written as 
\begin{equation}
 \psi_h(x_i, y_i) = (a_{i,1}\eta_i(x_i,y_i), \cdots,  a_{d,i}\eta_i(x_i,y_i)) = (a_{i,1}, \cdots, a_{d,i}) = (\gamma_1(t_i),\cdots,\gamma_d(t_i))
\end{equation}
for some $t_i \in [0, 2\pi]$.

We apply a relaxation procedure to find a stationary point $\varphi_h \in X_{\Gamma_h}^{tp}$. 
That is, to find a stationary point of $\D(\psi_h)$, only one vector $(a_{i,1}, \cdots, a_{i,d})$ is updated at each step. 
To apply the relaxation procedure to $\D(\psi_h)$, we have to consider two cases.

Case 1: $(x_i,y_i)$ is in the interior of $D$.  
In this case, because $\D(\psi_h)$ is a quadratic function with respect to $a_{i,k}$ and 
\[
   \frac{\partial\;}{\partial a_{i,k}}
   \left(\frac{1}{2}a_k^\top A a_k\right)
   = \sum_{j=1}^N \alpha_{ij}a_{j,k}, \quad k = 1, \cdots, d,
\]
we may use, for example, simple Gauss-Seidel iteration 
\[
   a_{i,k}^{(new)} := - \sum_{j=1, j\neq i}^N
   \frac{\alpha_{ij}}{\alpha_{ii}} a_{j,k}^{(old)}, \qquad
  k = 1, \cdots, d.
\]

Case 2: $(x_i,y_i)$ is on the boundary of $D$. 
In this case, the relaxation procedure becomes more complicated. 
We insert (2) into (1), and $\D(\psi_h)$ may be written as
\begin{equation*}
   F(t_i) := \D(\psi_h), \quad t_i \in [0, 2\pi]
\end{equation*}
in the relaxation step at a boundary node. 
We would like to find $t_i$ such that $F'(t_i) = 0$ at \textit{all} boundary nodes $(x_i,y_i)$. To
this end, we employ the Newton method
\begin{align*}
   t_i^{(new)} := t_i^{(old)} - \frac{F'(t_i^{(old)})}{F''(t_i^{(old)})}.
\end{align*}

At first, images of the points on $\partial D$ are distributed with
equal intervals on $\Gamma$  with respect to the positions of the three
fixed points.  In the optimization process, boundary point images
$\varphi_h(\mathcal{N}_{bdy})$ move rather freely on $\Gamma$.
As a result, we might have a FE minimal surface with a poor
approximation of $\Gamma$, if the cardinality of $\mathcal{N}_{bdy}$ is not 
large enough. Figure~\ref{fig:rockbad} illustrates such a situation. 
The parametric equation of the curve represented is 

\[\begin{cases}
	 x = (1+0.5\cos{3\theta})*\cos(\theta),\\
	 y = (1+0.5\cos{3\theta})*\sin(\theta),
\end{cases}
\]
and it looks like the one on the right side of Figure~\ref{fig:ncond}.

\begin{figure}
	\centering
	\subcaptionbox{Poor approximation of some parts of the domain (obtained in 806 iterations)\label{fig:rockbad}}
		{\includegraphics[width=0.4\textwidth]{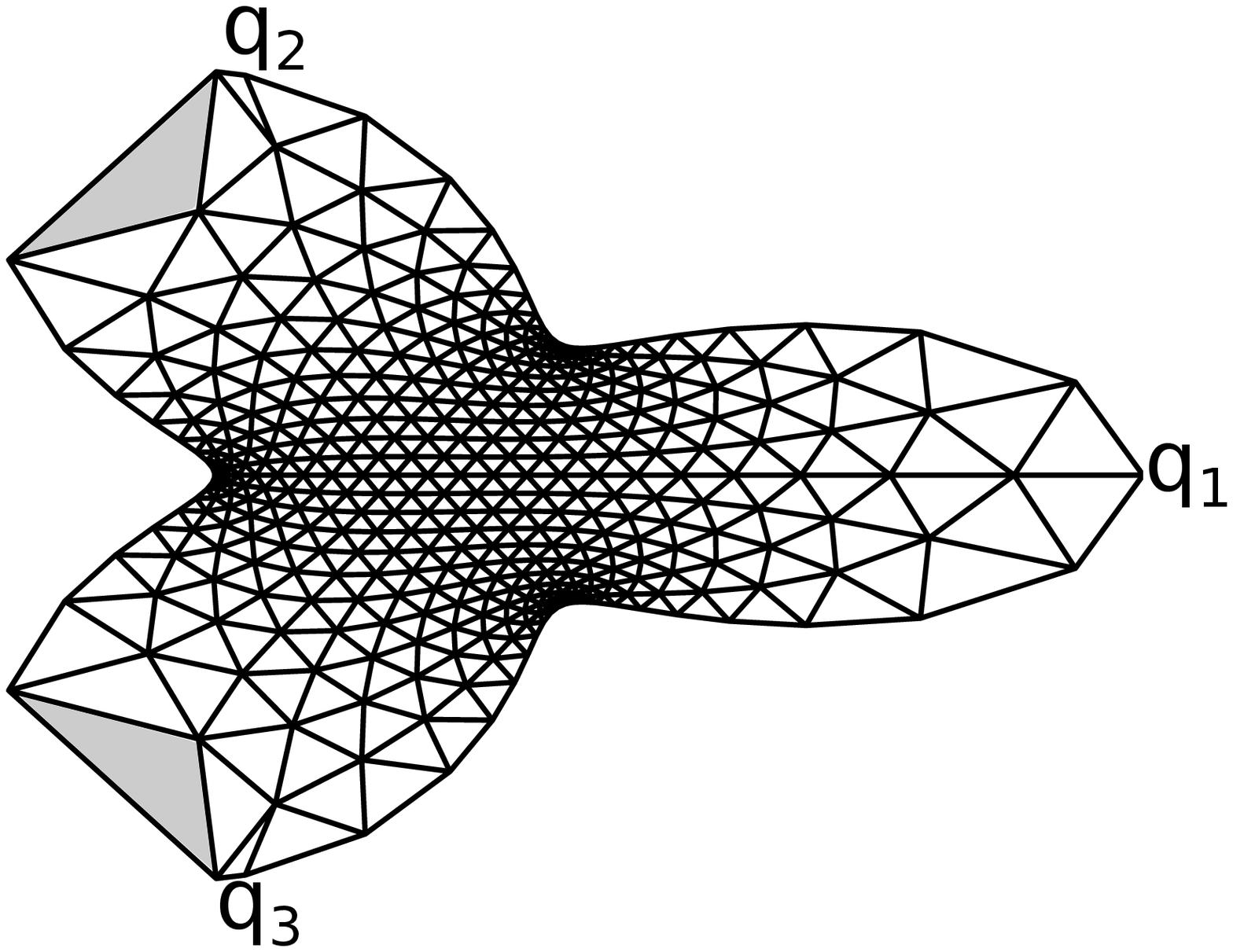}}\quad\quad
	\subcaptionbox{Better approximation obtained by adding four boundary points on defective triangles (obtained in 897 iterations)\label{fig:rockgood}}
	{\quad\includegraphics[width=0.4\textwidth]{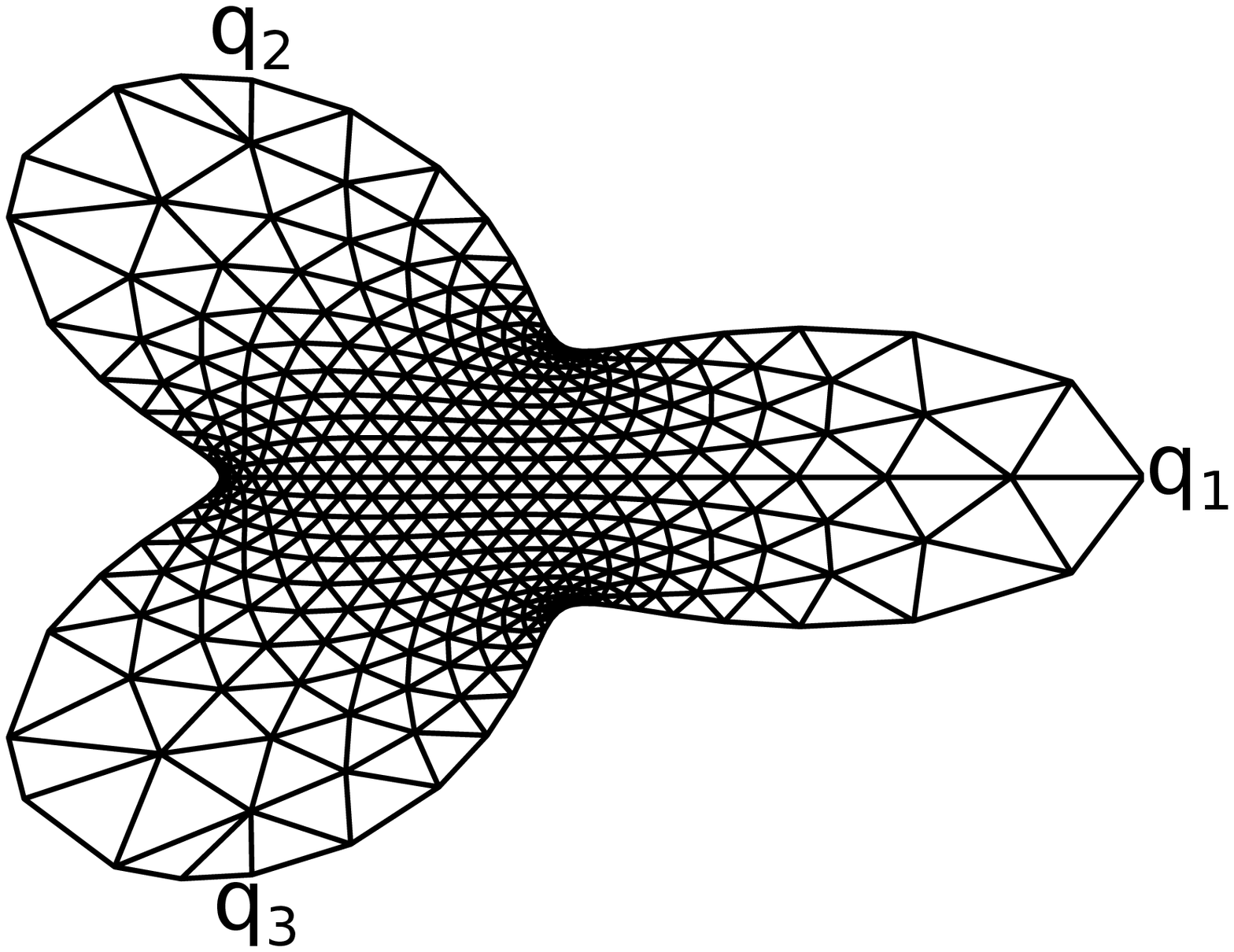}}
\caption{Comparison of two approximations of the same curve}
\label{fig:rocket}
\end{figure}

A simple naive way to obtain a better approximation is
of course to use a finer triangulation of the unit disk so that
$\varphi_h(\mathcal{N}_{bdy})$ provides a better approximation of
$\Gamma$.

Our simple and effective method consists of dynamical insertions of boundary
nodes on triangles whose images are ``defective''. A triangle is called a 
\emph{boundary triangle}, if two of its vertices belong to
$\mathcal{N}_{bdy}$, and the image of a boundary triangle is said to be
\emph{defective} if the distance (or the angle) between these two vertices
is much larger than the one of its neighbour triangles.
Let choose a positive integer $C$. Every $C$ iterations of the relaxation process, 
we check the quality of the boundary. 
When a defective triangle is identified, we split its inverse image into 
two smaller triangles by inserting a new boundary node at the middle of 
its boundary arc (the arc of circle between its two boundary nodes) and joining this node to the vertex facing it, before 
continuing the relaxation process.
Adding few boundary points to defective triangles during the relaxation process,
we can get a better approximation, as shown by Figure~\ref{fig:rockgood}. 

An even more problematic situation can occur, when the approximation completely
collapses, as shown by Figure~\ref{fig:rockcrash}. This case can be viewed as
an extreme case of the previous one. Such a situation would happen if a triangle becomes 
much larger than its neighbourhood, resulting in its collapse. It can therefore be overcome 
by the method described above, as shown by Figure~\ref{fig:rockbisect}. We see a limitation of 
such an adaptive bisection method. The computation is hindered by the insertion of too many points. 
Figure~\ref{fig:rockrest} shows the result of the same adaptive method but this time combined 
with a different refinement technique. A defective boundary triangle is split into four triangles 
by joining the middle points of its three edges. Such a refinement technique is often called 
\emph{regular refinement} and the previous one \emph{marked edge bisection} (see \cite[Section~4.1]{Verfurth}. 
The regular refinement has the advantage to produce triangles that are similar to the one from which they are created. 
On the other hand it involves a more complex and heavier computation as the technique introduces hanging 
nodes --- nodal points in the middle of an edge --- on the neighbours of the refined triangle. 
To preserve the triangulation, a hanging node is joined to the vertex facing it.

\begin{figure}
\centering
\begin{tabular}{cc}
\multirow{2}{*}{\subcaptionbox{The approximation collapses due to the choices of the fixed points and their images\label{fig:rockcrash}}
				{\includegraphics[width=0.4\textwidth]{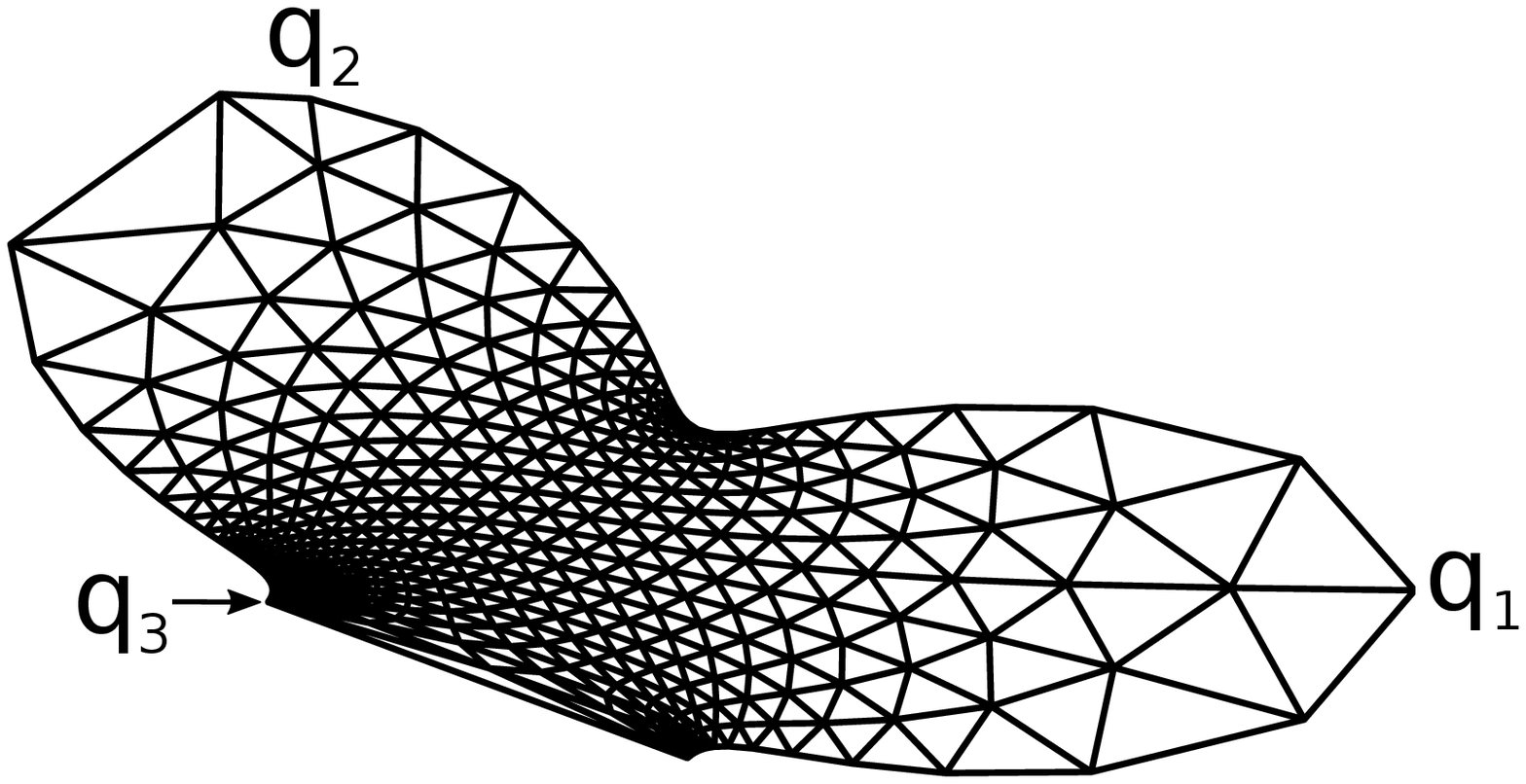}}} \quad \quad&
\makecell[r]{\subcaptionbox{The approximation is restored by bisecting defective triangles\label{fig:rockbisect}}
				{\includegraphics[width=0.4\textwidth]{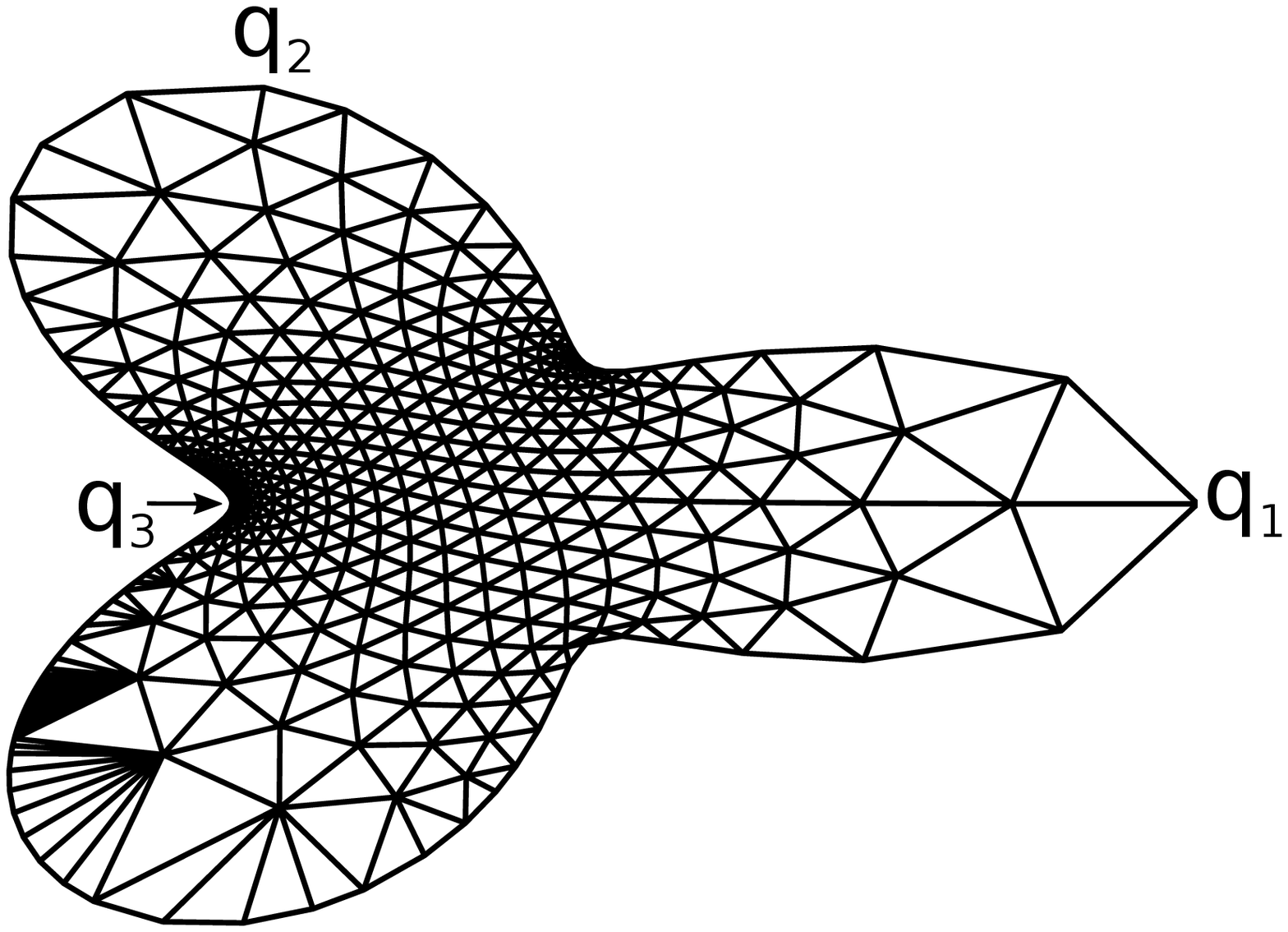}}}\\

&\makecell[r]{\subcaptionbox{The approximation is restored by regularly refining defective triangles\label{fig:rockrest}}
				{\includegraphics[width=0.4\textwidth]{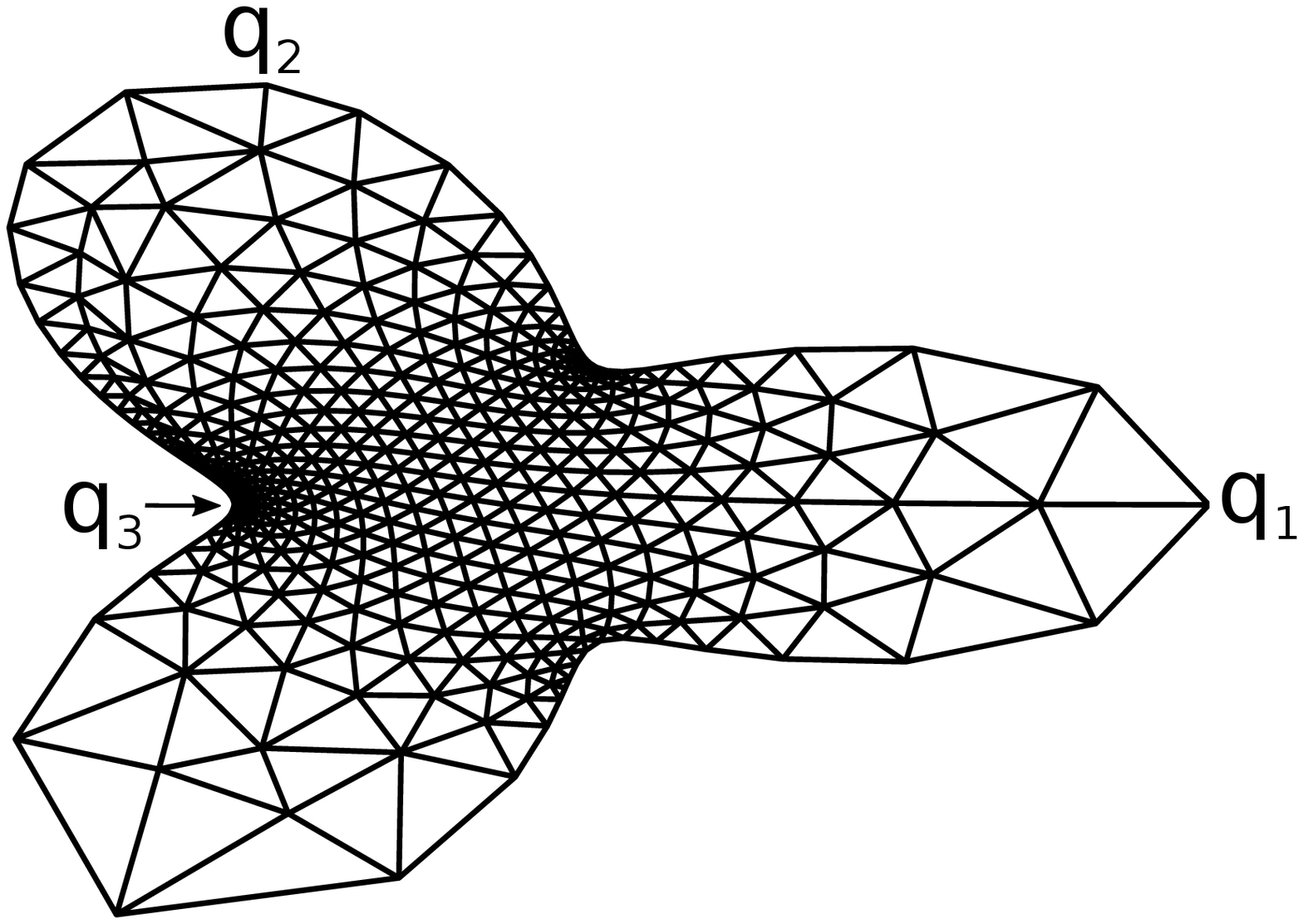}}}
\end{tabular}
\caption{An example of approximation that collapses and how it can be restored}
\label{fig:collapse}
\end{figure}

\subsection{A $3-$dimensional example} We provide another example of how the described method
allows to avoid the collapse of the approximation, this time in three dimensions. 
We use Triangulation~1 to map the unit circle to the curve defined by 
\[\begin{cases}
	x = \cos\theta(1+0.5*\sin(2\theta)),\\
	y = \sin\theta(1+0.5*\sin(2\theta)),\\
	z = 0.5\sin(3\theta).
\end{cases}
\]
We see in Figure~\ref{fig:3dbad} that the original approximation makes the boundary (the grey curve)
impossible to recognize. Figure~\ref{fig:3drest} shows the approximation obtained by 
adaptive regular refinements on defective boundary triangles. We may want to see how 
the approximation behaves if we choose different fixed points. Instead of $p_1$, $p_2$, and $p_3$,
we now choose $p_1$, $p'_2$, and $p'_3$ such as defined in Figure~\ref{fig:disk1}. 
Note that the choice of these points is arbitrary and serves no other purpose than to illustrate 
the fact that collapsed approximations are very common. 
Figures~\ref{fig:3d2bad} shows that, without inserting nodes, we get a 
better approximation than the previous one, but a part of the curve is still 
missing. By the same method, we once again avoid the collapse and get a good 
approximation, as shown by Figure~\ref{fig:3d2rest}. Note that even if we greatly improve the approximation, 
there may sometimes remain some areas that are still poorly approximated such as 
the sharp peak we can see in the figure. This can generally be solved by increasing 
the number of points in the triangulation (Figure~\ref{fig:3dnice} shows 
the result of our method applied to map Triangulation~2) or by locally refining, after the initial mapping, 
the problematic areas to better approximate the surface and its boundary (a posteriori refinement). 

\begin{figure}
	\centering
	\subcaptionbox{The original approximation has collapsed\label{fig:3dbad}}
		{\includegraphics[width=0.4\textwidth]{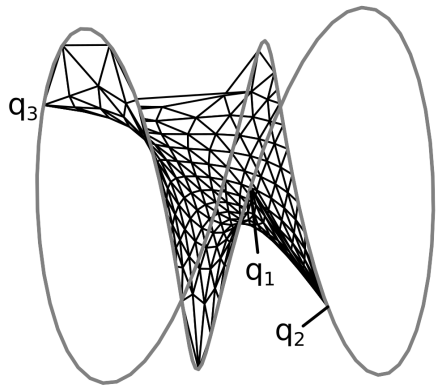}}\quad\quad
	\subcaptionbox{16 insertions to obtain a good approximation\label{fig:3drest}}
	{\includegraphics[width=0.4\textwidth]{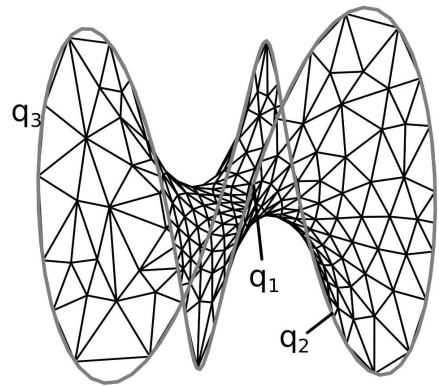}}\\
	
	\subcaptionbox{The same example with different fixed points\label{fig:3d2bad}}
		{\includegraphics[width=0.4\textwidth]{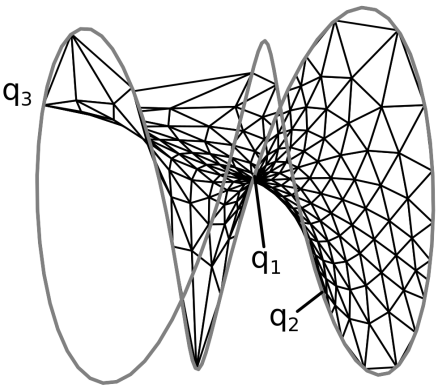}}\quad\quad
	\subcaptionbox{32 insertions to improve the approximation\label{fig:3d2rest}}
	{\includegraphics[width=0.4\textwidth]{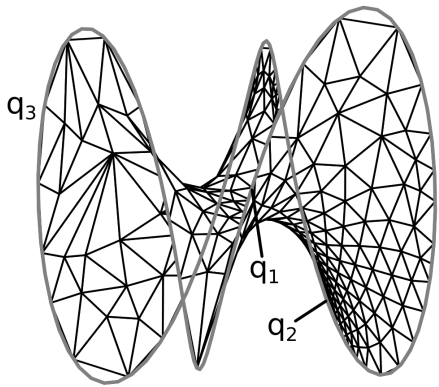}}\\
	
	\subcaptionbox{Increasing the number of nodes in the triangulation\label{fig:3dnice}}
	{\includegraphics[width=0.4\textwidth]{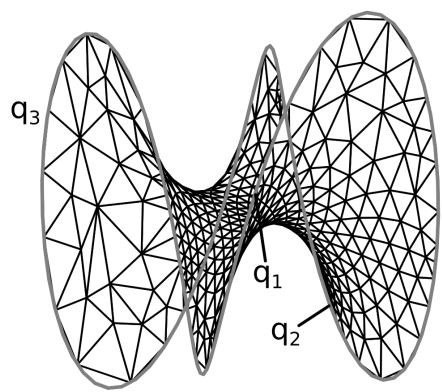}}
\caption{A $3-$dimensional example}
\label{fig:3d}
\end{figure}

\subsection{Approximating corners}
Such an adaptive mesh refinement method is particularly effective to handle curves 
with several corners. Approximations of corners by piecewise elements are often mediocre.
If the curve has less than four corners, these can be chosen as fixed images 
and the approximation can stay unspoiled. However, if the number of corners is higher,
one has to find a way to approximate them better. Let us take the example of the 
square $[-1,1]\times[-1,1]$.
One possibility is to use a smoothing function \cite{Tsuchiya4}, but we show in Figure~\ref{fig:square} 
how the approximation at the corner can be improved by the described method (here, we used bisections). 
The nodes on the edge of the bottom right corner have coordinates $(0.59,-1.0)$, $(1.0,-0.62)$ 
in Figure~\ref{fig:square_ori}, and $(0.86,-1.0)$, $(1.0,-0.88)$ in Figure~\ref{fig:square_ref}. 

\begin{figure}
	\centering
	\subcaptionbox{Approximation of the square\label{fig:square_ori}}
		{\includegraphics[width=0.4\textwidth]{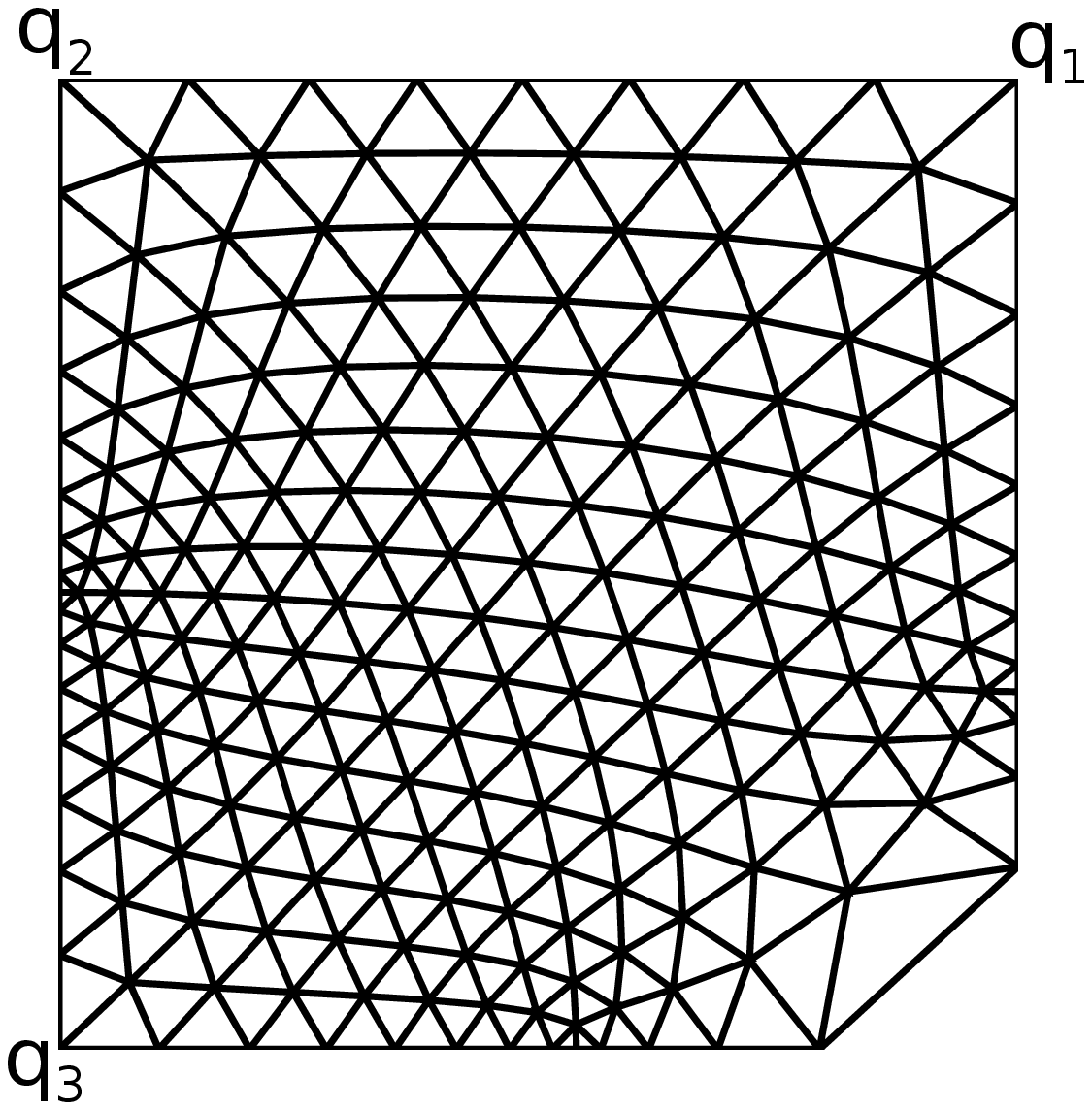}\quad}
	\subcaptionbox{4 points inserted\label{fig:square_ref}}
	{\quad\includegraphics[width=0.4\textwidth]{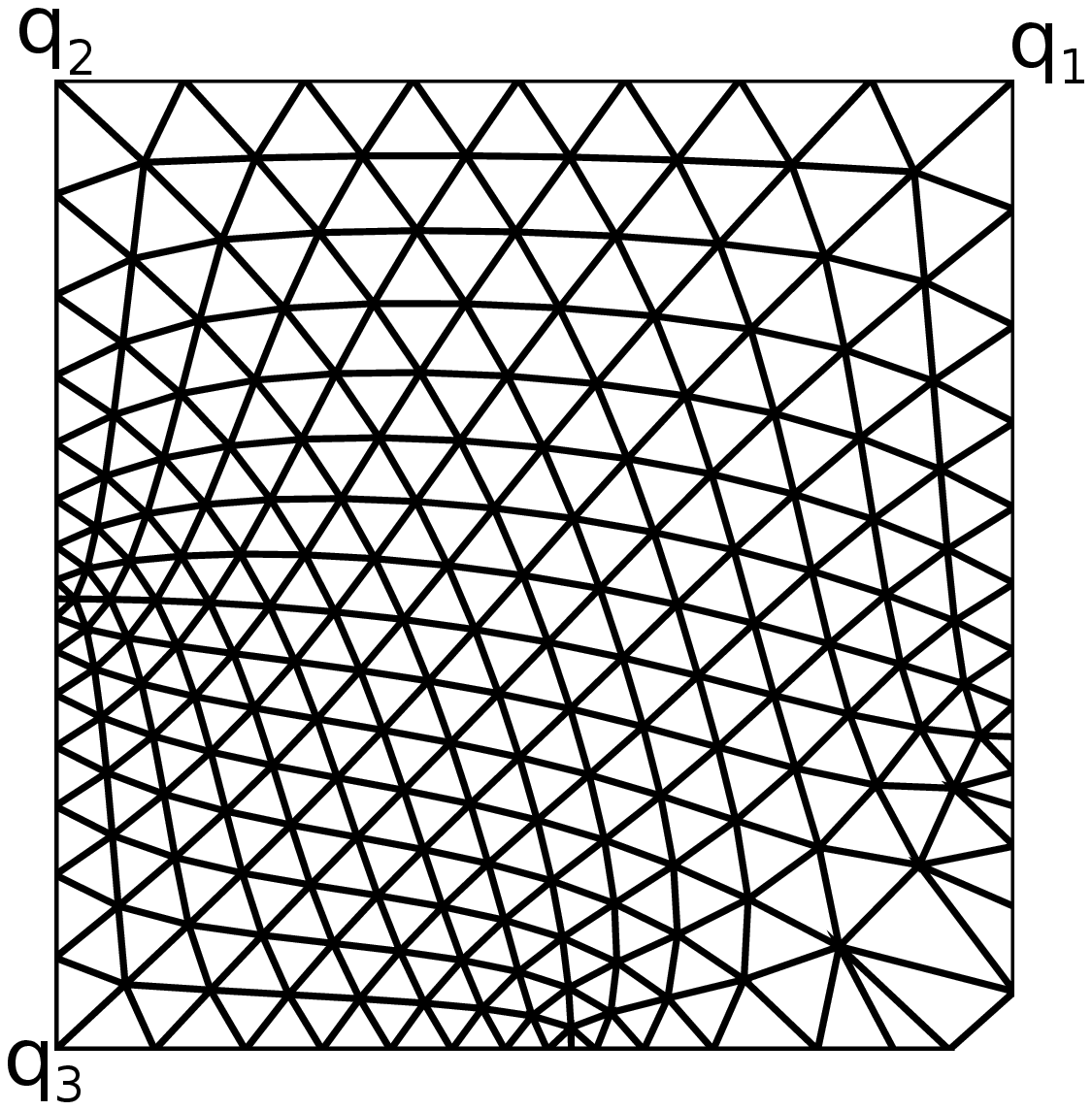}}
\caption{An approximation of the $[-1,1]\times[-1,1]$ square}
\label{fig:square}
\end{figure}

\FloatBarrier

\section{Partially free boundary}
\label{sec:pfb}
We now consider another problem that is to find a minimal surface with
a partially free boundary.  In this problem, the boundary to which we
map the circle consists of
the couple $\pfreeB$ where $\SSS$ is a
given closed surface in $\R^3$ and $\Gamma$ is a curve now connected to
$\SSS$ by two points $q_1$ and $q_3$. If the mappings presented
in Section~\ref{sec:minsurf} can be physically imagined as wire frames pulled
out of soapy water, the ones presented from now on can be seen as the film
created between the wire and an object connected to it.  Note that the surface of the
soapy water itself can be such an object. 
This case represents the situation \emph{while} we are pulling the frame out from the water.

Note that $\pfreeB$ can be modeled in
different ways.  One way is that $\Gamma$ is an open curve whose ending
points lie on $\SSS$. The other way is that one part of the
closed curve $\Gamma$ is ``merged'' into $\SSS$.
In this model, we shall refer to the points on that part of the curve as
\emph{surface points}. Note that we refer to $\SSS$ as
``surface'' since it is easily physically represented
by any surface, but it is of course in no case related to the minimal
surface we are looking for.

As an illustration, we give a numerical example of a minimal surface with partially free
boundary in Figure~\ref{fig:free}.   In the example, $\Gamma$ is
a part of a circle in $\R^3$ and $\SSS$ is (a convex bounded subset of)
a plane.

\begin{figure}
	\centering
	\subcaptionbox{$\Gamma$ is a small arc\label{fig:free1}}
		{\includegraphics[width=0.45\textwidth]{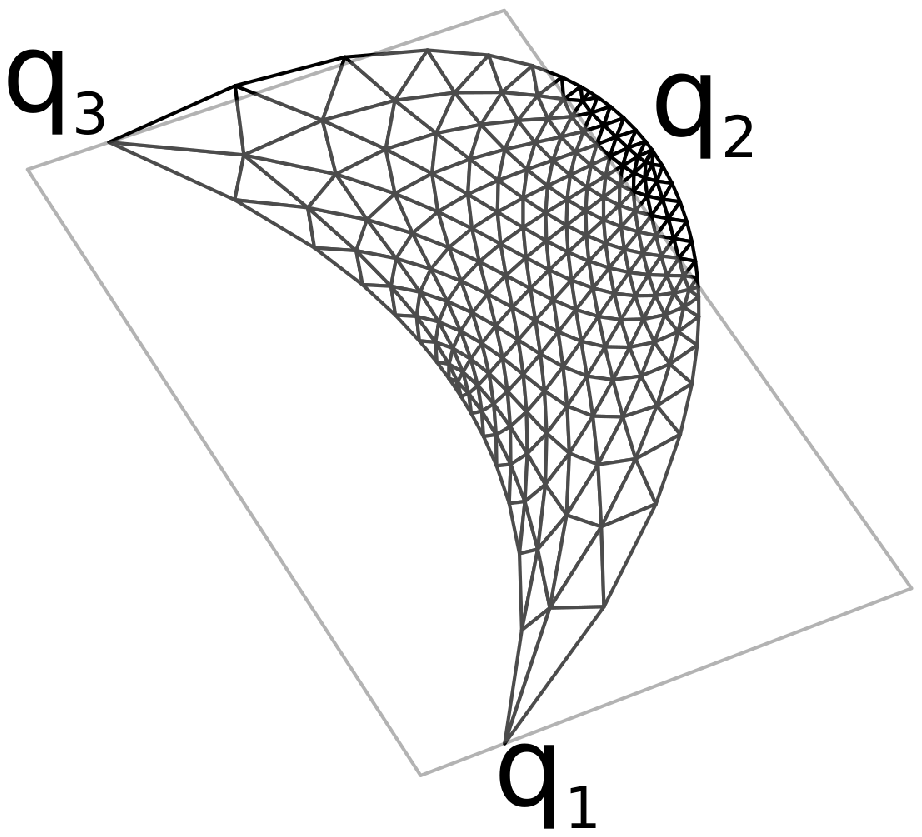}}
	\subcaptionbox{$\Gamma$ is a bigger arc of the same circle\label{fig:free2}}
		{\includegraphics[width=0.45\textwidth]{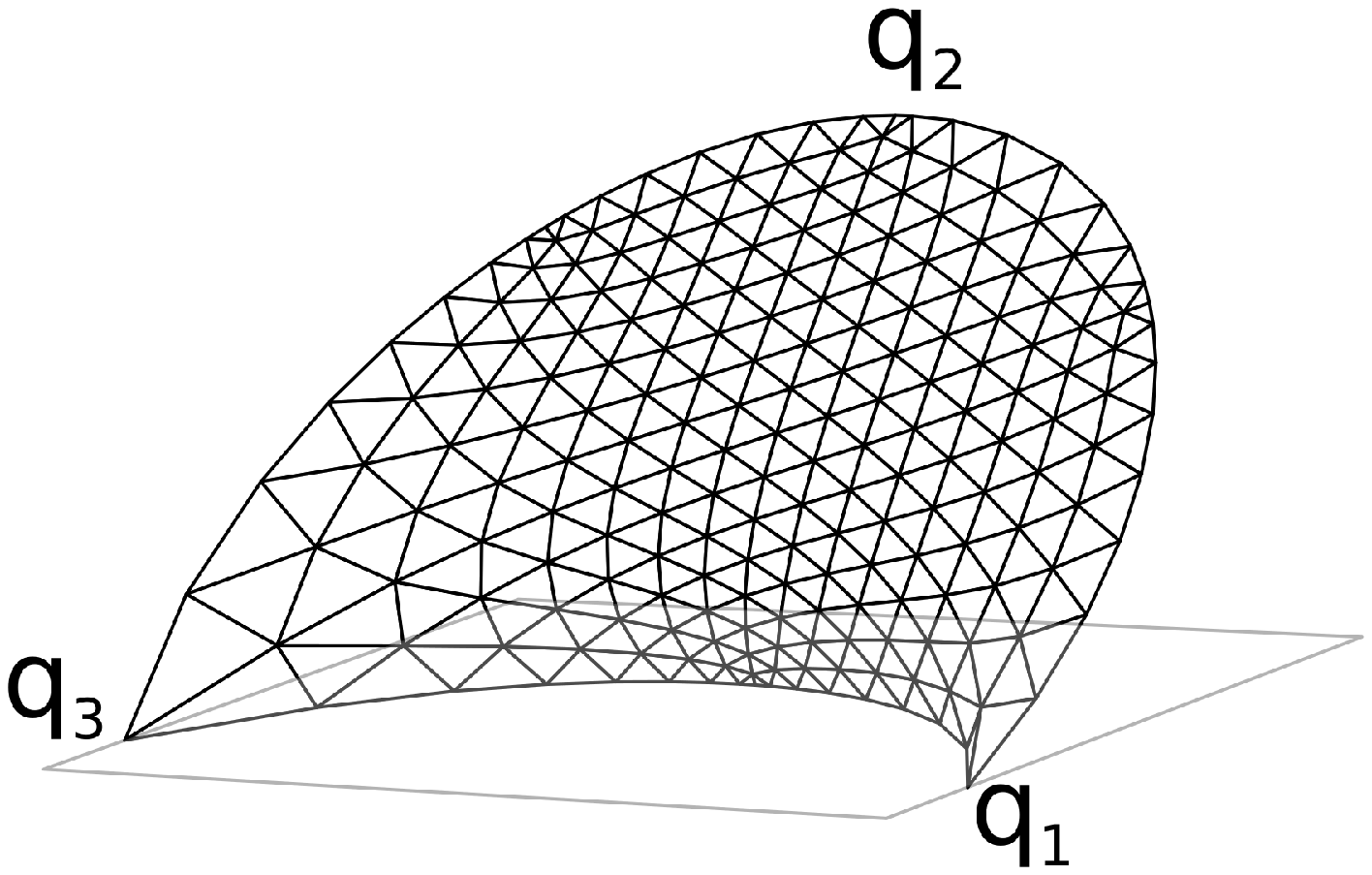}}
\caption{An example of minimal surface with partially free boundary}
\label{fig:free}
\end{figure}

We now define the problem rigorously.
Let $\Xi^1 \subset \partial D$ be the closed interval of $\partial D$
such that $p_2 \in \Xi^1$ and its end-points are $p_1$ and $p_3$.
Set $\Xi^2 := \partial D \backslash \Xi^1$.
We consider the following conditions for $\psi \in H^1(D;\R^3)$:
\begin{itemize}
 \item[(i)] $\psi|_{\Xi^2}(w) \in \SSS$ for almost all $w \in \Xi^2$.
 \item[(ii)] $\psi|_{\Xi^1}$ is continuous and monotone on $\Xi^1$
such that $\psi(\Xi^1) = \Gamma$ and $\psi(p_i) = q_i$, $i = 1, 3$.
\end{itemize}
Here, $\psi|_{\Xi^i}$ is the trace of $\psi$ on $\Xi^i$, $i=1,2$.
Then, the subset $X_{\pfreeB}$ is defined by
\begin{align*}
 X_{\pfreeB}:=\Bigl\{\psi \in H^1(D;\R^3)\Bigm|
  \text{$\psi$ satisfies (i) and (ii)} \Bigr\}.
\end{align*}

As is stated, we suppose that $\Gamma$ is connected to $\SSS$ by
$q_1$ and $q_3$, that is, $\Gamma \cap \SSS = \{q_1, q_3\}$.
Note that we have $\psi(p_i) = q_i$, $i=1,3$ for
$\psi \in X_{\pfreeB}$ by the definition.
We take $q_2 \in \Gamma$ and define
\begin{equation*}
 X_{\pfreeB}^{tp} := \Bigl\{\psi\in X_{\pfreeB}\Bigm| \psi(p_2)= q_2 \Bigr\}.
\end{equation*}
Then, as before, $\varphi \in X_{\pfreeB}^{tp}$ is a minimal surface
if and only if $\varphi$ is a stationary point of $\D$ in
$X_{\pfreeB}$.  If $X_{\pfreeB}^{tp}$ is not empty, there exists
a minimal surface in $X_{\pfreeB}^{tp}$ that attains the infimum of the
Dirichlet integral in $X_{\pfreeB}^{tp}$.  Again, that minimal surface
is called the \textit{Douglas-Rad\'o solution}. 
For the proof of existence of the Douglas-Rad\'o solution, see
\cite[Theorem~2, p.278]{DHS}.

The FE approximation of minimal surfaces with partially free
boundary is now almost obvious.
Define the subsets of $\mathcal{N}_{bdy}$
\[
   \mathcal{N}_{bdy}^1 := \{p \in \mathcal{N}_{bdy}
  : p \in \Xi^1\}, \qquad
   \mathcal{N}_{bdy}^2 := \mathcal{N}_{bdy} \backslash \mathcal{N}_{bdy}^1,
\]
and the discretizations of $X_{\pfreeB}$ and $X_{\pfreeB}^{tp}$ are defined by
\begin{alignat*}{2}
 X_{\pfreeB, h} & := \Bigl\{\psi\in (S_h)^d
  \Bigm| && \psi(\mathcal{N}_{bdy}^1) \subset \Gamma, \;
 \psi|_{\Xi^1} \text{ is $d$-monotone}, \\
  &  && \psi(\mathcal{N}_{bdy}^2) \subset \SSS, \;
   \psi(p_i)= q_i, i=1,3 \Bigr\}, \\
%
  X_{\pfreeB,h}^{tp} & := \Bigl\{\psi\in X_{\pfreeB}\Bigm|
   && \enspace \psi(p_2)= q_2 \Bigr\}.
\end{alignat*}
Stationary points in $X_{\pfreeB,h}^{tp}$ with respect to
the Dirichlet integral $\D$ are called $FE$
\emph{minimal surfaces with free boundary} on $\SSS$.
In particular, the minimizer of the Dirichlet integral in
$X_{\pfreeB,h}^{tp}$ is called
\textit{FE Douglas-Rad\'o solution}.

For convergence, we have the following theorem:

\begin{theorem}\label{Thm-pfreeB}
Suppose that $\pfreeB$ satisfies the following conditions:
\begin{itemize}
 \item $\Gamma$ is rectifiable,
 \item $\SSS$ is a bounded closed subset of a plane in $\R^3$,
 \item $\Gamma \cap \SSS = \{q_1, q_3\}$ and there is a rectifiable arc in $\SSS$ connecting $q_1$ and $q_3$.
\end{itemize}
Note that if $\pfreeB$ satisfies the above conditions,
$X_{\pfreeB}^{tp}$ is nonempty \cite[Theorem~2, p.278]{DHS}.
Suppose also that all the Douglas-Rad\'o solutions spanned in
$\pfreeB$ belong to $C(\overline{D};\R^3) \cap H^1(D;\R^3)$.

Let $\{\varphi_h\}_{h >0}$ be the sequence of FE Douglas-Rad\'o
solutions on triangulation $\T_h$ such that
$\varphi_h \in X_{\pfreeB,h}^{tp}$ and $|\T_h| \to 0$ as $h \to 0$.
Then, there exists a subsequence
$\{\varphi_{h_i}\} \subset X_{\pfreeB,{h_i}}^{tp}$ which converges to
one of the Douglas-Rad\'o solution $\varphi \in X_{\pfreeB}^{tp}$
spanned in $\pfreeB$ in the following sense$:$
\begin{equation}
    \lim_{h_i \to 0} \|\varphi_{h_i} - \varphi\|_{H^1(D;\R^3)}
     = 0,
   \label{convergence3}
\end{equation}
and if $\varphi \in W^{1,p}(D;\R^3)$, $p > 2$, then
\begin{equation}
 \lim_{h_i \to 0} \|\varphi_{h_i} - \varphi\|_{C(D\cup \mathscr{C};\R^3)} = 0,
   \label{convergence4}
\end{equation}
where $\mathscr{C} \subset \Xi^1$ is an arbitrary open arc contained in $\Xi^1$.
If the Douglas-Rad\'o solution is unique, then $\{\varphi_h\}$ converges
to $\varphi$ in the sense of \eqref{convergence3} and \eqref{convergence4}.
\end{theorem}

\begin{proof}
Let $\varphi \in X_{\pfreeB}^{tp}$ be one
of the Douglas-Rad\'o solution and let $\eta_h \in (S_h)^3$ be the
FE solution such that
\begin{align*}
   \int_D \nabla \eta_h\cdot \nabla \mathbf{v}_h \dd x = 0, \quad
   \forall \mathbf{v}_h \in (S_h)^3, \quad
    \mathbf{v}_h|_{\partial D} = 0,
\end{align*}
with $\varphi(x) = \eta_h(x)$ for all $x \in \mathcal{N}_{bdy}$.

With the inequalities \eqref{eq:ineq}, we can modify the proof of
\cite[Theorem3.2.3]{Ciarlet} and prove that $\lim_{h \to 0}\|\varphi - \eta_h\|_{H^1(D)} = 0$.


By the definition of FE Douglas-Rad\'o
solutions, we have $\D(\varphi_h) \le \D(\eta_h)$, and hence
$\D(\varphi_h)$ is uniformly bounded.  Also,
$\|\varphi_h\|_{L^2(D)}$ is bounded because of the (discretized) maximum
principle.  Thus, $\|\varphi_h\|_{H^1(D)}$ is bounded, and there exists
a subsequence $\{\varphi_{h_i}\}$ that 
converges weakly in $H^1(D;\R^3)$ to some $\psi \in H^1(D;\R^3)$.  In
the following, we show that $\psi$ belongs to $X_{\pfreeB}^{tp}$ and is
one of the Douglas-Rad\'o solution.

By the lower-semicontinuity of the Dirichlet integral with respect to
weak convergence in $H^1(D;\R^3)$ that is shown
in the proof of Theorem~1 \cite[p.276]{DHS}, we have
\[
   \D(\psi) \le \liminf_{h_i \to 0} \D(\varphi_{h_i})
   \le \lim_{h_i \to 0} \D(\eta_{h_i}) = \D(\varphi).
\]
The last equality follows from the fact
$\lim_{h \to 0} \|\varphi - \eta_h\|_{H^1(D)} = 0$.
Hence, if $\psi \in X_{\pfreeB}^{tp}$, we conclude that
$\psi$ is one of the Douglas-Rad\'o solutions.

Because $\varphi_{h_i}$ converges weakly in $H^1(D;\R^3)$ to $\psi$,
we have
\begin{align}
   \lim_{h_i \to 0} \|\varphi_{h_i} - \psi\|_{L^2(D;\R^3)} = 0,
  \qquad
  \lim_{h_i \to 0} \|\varphi_{h_i}|_{\partial D}
     - \psi|_{\partial D}\|_{L^2(\partial D;\R^3)} = 0
  \label{Rellish}
\end{align}
by Rellish's theorem, where
$\varphi_{h_i}|_{\partial D}$ and $\psi|_{\partial D}$ are
the traces of $\varphi_{h_i}$ and $\psi$ on $\partial D$.

$\varphi_{h_i}(\Xi^1)$ is a polygonal curve approximating $\Gamma$ ***AND***  Applying \cite[Lemma~3]{Tsuchiya3} to $\varphi_{h_i}(\Xi^1)$,
we notice that $\varphi_{h_i}(\Xi^1)$ converges uniformly to $\Gamma$
as $h_i \to 0$.  Hence, $\psi|_{\Xi^1}$ is continuous and monotone
such that $\psi(\Xi^1)$ with $\psi(p_i) = q_i$, $i = 1,2,3$.

By the assumptions on $\SSS$, $\SSS$ is a subset of a plane and $\varphi_{h_i}|_{\Xi^2}$ is a polygonal curve satisfying $\varphi_{h_i}(\mathcal{N}_{bdy}^2) \subset \SSS$. 
Therefore, we have $\varphi_{h_i}|_{\Xi^2} \subset \SSS$. 
Because of \eqref{Rellish}, we have $\psi|_{\Xi^2}(w) \in \SSS$ for almost all $w \in \Xi^2$.
Therefore, we conclude with $\psi \in X_{\pfreeB}^{tp}$ and the proof is completed.  $\square$
\end{proof}

The condition on $\SSS$ is rather restrictive.  The authors hope
the condition will be weakened in future.

\FloatBarrier

\section{Data structures and algorithm}
\label{sec:data}
The implementation of data structures for finite element methods and mesh refinement 
depends on several things, among which:
\begin{itemize}
	\item The computational object at the center of the computation. There are usually two choices: 
	nodes or elements of the triangulation. The main difference is that the latter being composed of the former, 
	either way, one need a strategy to go from nodes to elements and vice versa.
	\item The type of elements we deal with. In this paper, they are triangular elements.
	\item The refinement strategies, that is to say when and how refinement(s) occur: during the computation or a posteriori? Is it a bisection or a regular refinement?
\end{itemize}
Therefore, we describe, as an example, the choice we made for our own computation but it is left to the 
discretion of the reader to see if this fits their needs.\\

Our method is element-oriented (i.e.\ our computation treats triangles as the main objects). 
The nodes of the triangulation are ordered in a certain way. Accordingly, to each node corresponds an index. 
We refer to this index as the \emph{global index}.
However, since our method is element-oriented, we store this information as part of the information about elements. 
Therefore, we define an array \texttt{elements} describing the list of elements, in which, for $j=1,2,3$, 
\begin{center} \texttt{elements[i,j] =} global index of the $j$th node of element $i$. \end{center}
As the boundary is a special part, different from its mapped interior, we establish an ordered list of boundary nodes, so we have
\begin{center} \texttt{boundary[i] =} global index of the $i$th node of the boundary. \end{center}
Hence we can say that the node with global index \texttt{boundary[i]} has \emph{boundary index} $i$.

To indicate the location of the node with global index $i$, we use an array \texttt{status}, whose size is the number of nodes in the triangulation, 
and is at first initialized by
\begin{align*}\texttt{status[i] =} \begin{cases}
               0 & \mbox{if node } i \mbox{ is an interior node},\\
               \mbox{its boundary index} & \mbox{if } i \mbox{ is a boundary point},
            \end{cases}
\end{align*}
which is modified by 
\begin{align*}\texttt{status[i] =} \begin{cases}
									-\texttt{status[i]} & \mbox{if } i \mbox{ is a fixed point},\\
									\texttt{status[i]} + \mathbf{card}(\mathcal{N}_{bdy}) & \mbox{if } i \mbox{ is a free boundary node}.
            \end{cases}
\end{align*}

For some refinement algorithms, error estimation, or problems that ask, for example, 
search of triangles having common nodes, it is very useful to implement the relation 
between neighbour triangles. This is the case for the regular refinement strategy. 
To do so, we use an array \texttt{neighbours} where, for $j = 1,2,3$, 
\begin{align*}\texttt{neighbours[i,j] =} \begin{cases}
								k & \mbox{if triangle } k \mbox{ is neighbour to triangle } i \\ 
								& \mbox{and } j\mbox{th node of } i \mbox{ is not common to } i \mbox{ and } k,\\
								-1 & \mbox{if the edge opposite to } j \mbox{th node of } i \mbox{ is part of the boundary}.
								\end{cases}
\end{align*}
The neighbourhood relation is illustrated by Figure~\ref{fig:neighbour}.

Finally, we need a structure for the matrix to which we apply the relaxation procedure. 
For the problem of this paper, a large part of this matrix is filled by zeros. 
Consequently, a list of lists that contains only the necessary values for each element is preferable to a full $2$-dimensional matrix.

\begin{figure}
\centering
	\includegraphics[width=0.75\textwidth]{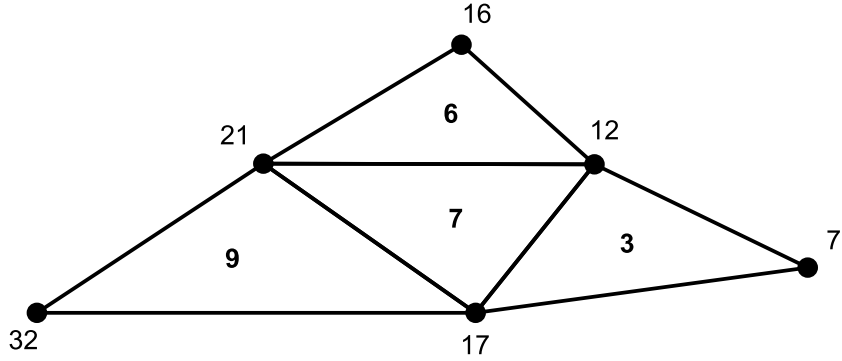}
	\caption{The neighbourhood of triangle $7$. elements$[7] = [21,17,12]$ so neighbours$[7] = [3,6,9]$}
\label{fig:neighbour}
\end{figure}

Algorithm~\ref{alg:rp} presents an overview of how the method is applied.

\begin{algorithm}
\renewcommand{\algorithmicrequire}{\textbf{Input:}}
\renewcommand{\algorithmicensure}{\textbf{Output:}}

\caption{Relaxation process}
\label{alg:rp}

\begin{algorithmic}[1]
\Require {elements, boundary, node\_status, neighbours, functions describing $\Gamma$, global matrix}
\Statex\hrulefill

\State $count \gets 0$
\State $C \gets$ number of iterations between two quality checks of the boundary
\Repeat 
	\If {$count \equiv 0 \pmod C$}
		\State \Call{checkTriangles}{}
	\EndIf
	\For {$i:1 \to $ number of nodes}
		\If {$status[i] < 0$} \Comment fixed point
			\State \textbf{continue}
		\ElsIf {$status[i] == 0$} \Comment inner point
			\State Relaxation by Gauss-Seidel method
		\ElsIf {$status[i] > 0$} \Comment boundary node
			\State Relaxation by Newton method	
			\If {$status[i] > \mathbf{card}(\mathcal{N}_{bdy})$} \Comment free boundary node
				\State Project node on surface.
			\EndIf
		\EndIf
	\EndFor
	\State $count \gets  count+1$
\Until being under the relaxation tolerance threshold

\Statex

\Procedure{checkTriangles}{}
	\For {$i:1 \to$ number of boundary nodes}
		\If {Criterion for insertion is met}
			\State $ele \gets$ element containing $boundary[i]$ and $boundary[i+1]$
			\State \Call{Bisection}{$ele$}
			\State \textbf{or}
			\State \Call{RegularRefinement}{$ele$}
		\EndIf
	\EndFor
\EndProcedure

\hrulefill
\Ensure Finite element approximation of the desired minimal surface

\end{algorithmic}
\end{algorithm}

\FloatBarrier

\section{Conclusion}
The refinement method presented in this paper has two main advantages.

First, it reduces the impact of the initial choices of the fixed points and their images. 
One can choose those in any way they want, the approximation obtained 
will never be collapsed and it will be close to the original contour. 

Second, it allows us to reduce the number of nodes in the initial triangulation. 
As examples showed, it is not always necessary to substantially increase the number of nodes of the initial triangulation to obtain a good enough approximation. 
The described method refines the boundary very locally so only a necessary amount of points is inserted. 

In the future, we will study how this simple method 
can be helpful to tackle more complex problems.
We extended our study of finite element approximations to minimal surfaces with partially free boundary
and we hope that the restrictions for the proof of convergence can be reduce.
We will study partially free boundary problems with more complex surfaces than planar ones 
and we will challenge the Douglas-Plateau problem, 
where one looks for a minimal surface spanned in a system 
of several Jordan curves, such as a pyramid or a cube.


\bibliographystyle{spmpsci}      
\bibliography{biblio}   

\end{document}